\documentclass[a4paper,12pt]{article} 
\usepackage[utf8]{inputenc}
\usepackage[T1]{fontenc}
\usepackage{amsmath,amssymb,amsthm}
\usepackage[]{authblk}
\usepackage{bbm}
\usepackage{mathrsfs}
\usepackage[numeric,initials,nobysame,msc-links,abbrev]{amsrefs}
\renewcommand{\eprint}[1]{\href{https://arxiv.org/abs/#1}{arXiv:#1}}
\newcommand{\pageafter}[1]{#1~pp.}
\BibSpec{article}{%
+{} {\PrintAuthors} {author}
+{,} { \textit} {title}
+{.} { } {part}
+{:} { \textit} {subtitle}
+{,} { \PrintContributions} {contribution}
+{.} { \PrintPartials} {partial}
+{,} { } {journal}
+{} { \textbf} {volume}
+{} { \PrintDatePV} {date}
+{,} { \issuetext} {number}
+{,} { \pageafter} {pages}
+{,} { } {status}
+{,} { \PrintDOI} {doi}
+{,} { available at \eprint} {eprint}
+{} { \parenthesize} {language}
+{} { \PrintTranslation} {translation}
+{;} { \PrintReprint} {reprint}
+{.} { } {note}
+{.} {} {transition}
+{} {\SentenceSpace \PrintReviews} {review}
}
\BibSpec{collection.article}{%
+{} {\PrintAuthors} {author}
+{,} { \textit} {title}
+{.} { } {part}
+{:} { \textit} {subtitle}
+{,} { \PrintContributions} {contribution}
+{,} { \PrintConference} {conference}
+{} {\PrintBook} {book}
+{,} { } {booktitle}
+{,} { \PrintDateB} {date}
+{,} { \pageafter} {pages}
+{,} { } {status}
+{,} { \PrintDOI} {doi}
+{,} { available at \eprint} {eprint}
+{} { \parenthesize} {language}
+{} { \PrintTranslation} {translation}
+{;} { \PrintReprint} {reprint}
+{.} { } {note}
+{.} {} {transition}
+{} {\SentenceSpace \PrintReviews} {review}
}
\usepackage{verbatim}
\usepackage{mathtools}
\usepackage{accents}
\usepackage{color}
\usepackage{bbm}
\usepackage{enumitem}
\usepackage[capitalise]{cleveref}

\setlist[itemize]{leftmargin=*}
\setlist[enumerate]{leftmargin=*,label=(\arabic*),ref=(\arabic*)}

\newtheorem{thm}{Theorem}
\crefname{thm}{Theorem}{Theorems}
\newtheorem{cor}[thm]{Corollary}
\crefname{cor}{Corollary}{Corollaries}
\newtheorem{lem}[thm]{Lemma}
\crefname{lem}{Lemma}{Lemmas}
\newtheorem{prop}[thm]{Proposition}
\crefname{prop}{Proposition}{Propositions}

\crefname{conj}{Conjecture}{Conjectures}
\newtheorem{ques}[thm]{Question}
\crefname{ques}{Question}{Questions}
\theoremstyle{definition}

\crefname{defn}{Definition}{Definitions}
\newtheorem{rem}[thm]{Remark}
\crefname{rem}{Remark}{Remarks}

\crefname{ex}{Example}{Examples}

\crefname{obs}{Observation}{Observations}

\crefname{claim}{Claim}{Claims}

\crefname{ass}{Assumption}{Assumptions}

\numberwithin{thm}{section}

\newcommand{\cB}{\ensuremath{\mathcal B}}

\newcommand{\cU}{\ensuremath{\mathcal U}}

\newcommand{\cX}{\ensuremath{\mathcal X}}

\newcommand{\bbE}{{\ensuremath{\mathbb E}} }

\newcommand{\bbN}{{\ensuremath{\mathbb N}} }

\newcommand{\bbP}{{\ensuremath{\mathbb P}} }

\newcommand{\bbR}{{\ensuremath{\mathbb R}} }

\newcommand{\bbZ}{{\ensuremath{\mathbb Z}} }

\let\oldd\d
\renewcommand{\d}{{\ensuremath{\delta}}}

\newcommand{\e}{{\ensuremath{\varepsilon}}}

\newcommand{\h}{{\ensuremath{\eta}}}
\let\oldk\k
\renewcommand{\k}{{\ensuremath{\kappa}}}

\let\oldl\l
\renewcommand{\l}{{\ensuremath{\lambda}}}
\let\oldL\L
\renewcommand{\L}{{\ensuremath{\Lambda}}}
\newcommand{\m}{{\ensuremath{\mu}}}
\newcommand{\n}{{\ensuremath{\nu}}}
\let\oldo\o
\renewcommand{\o}{{\ensuremath{\omega}}}
\let\oldO\O
\renewcommand{\O}{{\ensuremath{\Omega}}}

\let\oldr\r
\renewcommand{\r}{{\ensuremath{\rho}}}

\let\oldt\t
\renewcommand{\t}{{\ensuremath{\tau}}}
\let\oldu\u
\renewcommand{\u}{{\ensuremath{\upsilon}}}
\newcommand{\x}{{\ensuremath{\xi}}}

\renewcommand{\>}{\rangle}

\newcommand{\supp}{\operatorname{supp}}
\newcommand{\Int}{\operatorname{Int}}
\newcommand{\var}{\operatorname{Var}}

\newcommand{\1}{{\ensuremath{\mathbbm{1}}} }

\newcommand{\sD}{\ensuremath{\mathscr D} }
\newcommand{\sU}{\ensuremath{\mathscr U} }
\newcommand{\fD}{\ensuremath{\mathfrak D} }

\newcommand{\bzero}{\ensuremath{\mathbf 0} }
\newcommand{\bone}{\ensuremath{\mathbf 1} }
\newcommand{\pc}{\ensuremath{p_{\mathrm{c}}} }
\newcommand{\qc}{\ensuremath{q_{\mathrm{c}}} }
\newcommand{\tu}{\ensuremath{\widetilde\u} }

\renewcommand{\th}{\ensuremath{\widetilde{\h}}}
\DeclareMathAccent{\wtilde}{\mathord}{largesymbols}{"65}

\renewcommand{\leq}{\leqslant}

\renewcommand{\le}{\leqslant}
\renewcommand{\ge}{\geqslant}
\renewcommand{\to}{\rightarrow}

\begin{document}
\title{Bootstrap percolation, probabilistic cellular automata and sharpness}
\author{Ivailo Hartarsky\thanks{\textsf{hartarsky@ceremade.dauphine.fr}}}
\affil{CEREMADE, CNRS, Universit\'e Paris-Dauphine, PSL University\protect\\75016 Paris, France}
\date{\vspace{-0.25cm}\today}
\maketitle
\vspace{-0.75cm}
\begin{abstract}
We establish new connections between percolation, bootstrap percolation, probabilistic cellular automata and deterministic ones. Surprisingly, by juggling with these in various directions, we effortlessly obtain a number of new results in these fields. In particular, we prove the sharpness of the phase transition of attractive absorbing probabilistic cellular automata, a class of bootstrap percolation models and kinetically constrained models. We further show how to recover a classical result of Toom on the stability of cellular automata w.r.t.\ noise and, inversely, how to deduce new results in bootstrap percolation universality from his work.
\end{abstract}


\section{Introduction}
\label{sec:intro}
There are numerous links between probabilistic cellular automata (PCA) \cite{Louis18} and percolation (see \cref{subsec:models} for the definitions of our models of interest) \cite{Grimmett99}. In the case of additive PCA this link is very apparent, since they may be viewed as oriented percolation models (see e.g.\ \cite{Hartarsky21GOSP}). Moreover, percolation is often used as a reference model for comparison in more complex cases (see e.g.\ \cite{Marcovici19}). 

Our first goal will be to import a recent technique \cite{Duminil-Copin19} for proving the sharpness of phase transitions from percolation to the setting of attractive PCA. This allows us to establish that they all `die out' exponentially fast throughout their `subcritical' phase. This comes to complement a classical result of Bezuidenhout and Gray \cite{Bezuidenhout94} showing that a certain `supercritical' phase is also well-behaved.

Besides PCA, our other main motivation for pursuing this result comes from bootstrap percolation (BP). We establish a correspondence between the two, so as to deduce exponential decay of the probability of remaining healthy above criticality previously conjectured for a class of BP models. This also has implications for related kinetically constrained models (KCM), taking into account previous work of the author \cite{Hartarsky21}.

Finally, we show other uses of the correspondence between PCA and BP. Namely, it provides an equivalence between the non-triviality of the phase transition of certain BP models and the stability w.r.t.\ noise of certain deterministic cellular automata (CA). The former was studied recently in the framework of BP universality by  Bollob\'as, Smith and Uzzell \cite{Bollobas15}, Balister, Bollob\'as, Przykucki and Smith \cite{Balister16} and Balister, Bollob\'as, Morris and Smith \cite{Balister22}, while the latter was investigated over four decades ago by Toom \cite{Toom80} and subsequently by a number of authors \cites{Gacs21,Bramson91,Lebowitz90,Gray99,Berman88,Swart22}. Bridging their viewpoints yields results in both directions.

\subsection{Models}
\label{subsec:models}
\subsubsection{General setting}
\textbf{Convention} As it is common in set systems, we will denote points with lower case letters, sets of points with upper case ones (or with Greek lower case letters), families of such sets by capital calligraphic letters, systems of such families with capital script letters and in the rare case of classes of such systems, we will use capital fraktur letters. Here and below we use the words `set', `family', `system' and `class' as synonyms, but we will reserve their usage to the corresponding levels as much as possible.

Throughout, unless otherwise stated, we fix an integer dimension $d\ge 1$ and \emph{range} $r\in[1,\infty)$. We set $R=([-r,r]^{d}\times [-r,0))\cap\bbZ^{d+1}$, so as to allow non-zero memory. Models \emph{without memory} will be defined identically, taking $R=([-r,r]^d\times\{-1\})\cap\bbZ^{d+1}$. A \emph{configuration} is any element of $\O=\{0,1\}^{\bbZ^{d}\times([-r,0)\cap\bbZ)}$. For any set $X$, we identify any $\h\in\{0,1\}^X$ with a subset of $X$ in the natural way. An \emph{up-set} of a partially ordered set $(P,\ge)$ is a subset $U\subset P$ such that for any $(p,u)\in P\times U$ such that $p\ge u$ we have $p\in U$. We similarly define \emph{down-sets} (which are the complements of up-sets). Let $\sU$ denote the system of all up-families of $\O_R:=\{0,1\}^R$ equipped with the partial order $\h\ge\o$ if $\h_x\ge \o_x$ for all $x\in R$ (that is, if $\h\supset\o$). Note that $\varnothing\in\sU$ and $\O_R\in\sU$. We will further need to consider the class $\fD$ of down-sets of the partially ordered system $(\sU,\supset)$.

An \emph{attractive PCA}\footnote{More generally, PCA are defined identically by a rates measure supported not only on $\sU$, but on the entire power set of $\O_R$. However, attractiveness is essential for everything we will say, so we restrict directly to the relevant setting. See e.g.\ \cite{Salo21} for problems arising immediately without this assumption.} will be defined by the \emph{rates} $\u(\{\cU\})\in[0,1]$ for $\cU\in\sU$. We require $\sum_{\cU\in\sU}\u(\{\cU\})=1$ and view $\u$ as a probability measure on $\sU$. We say simply \emph{attractive CA}, if $\u$ is a Dirac measure. An attractive PCA is said to be \emph{additive} if $\u(\{\cU\})=0$ unless $\cU$ is generated by singletons, that is, there exists $X\subset R$ (possibly empty) such that $\cU=\{Y\subset R:X\cap Y\neq\varnothing\}$. We further say that it is \emph{absorbing} if $\u(\{\O_R\})=0$, which will be equivalent to saying that the $\bzero$ configuration (which we identified with the set $\varnothing$) is an absorbing state. 
Hence, attractive PCA identify with a finite dimensional simplex equipped with the standard topology (that is, the topology of weak convergence of the corresponding $\u$ measures) and similarly for additive, absorbing attractive and absorbing additive ones.

Given the rates $\u$, we can construct the associated finite memory Markov chain on $\{0,1\}^{\bbZ^d}$ graphically as follows. In words, at each time step $t\in\bbZ$, the state of each site $x\in\bbZ^{d}$ becomes $1$ if and only if the restriction of the current configuration to $R+(x,t)$ belongs to a randomly chosen up-family with law $\u$. More formally, let $(\cU_{x,t})_{(x,t)\in\bbZ^{d+1}}$ be an i.i.d.\ random field of up-families with law $\u$. Given the state of the PCA $\h$ at times $t-r,\dots,t-1$ (which form a configuration) and $x\in\bbZ^d$, we define 
\begin{equation}
\label{eq:def:PCA}
\h_{x}(t)=\begin{cases}1&\text{if }\left\{(y,s)\in R:\h_{x+y}(t+s)=1\right\}\in \cU_{x,t}\\
0&\text{otherwise}.\end{cases}
\end{equation}
This defines the trajectory of the PCA, given the initial state and the up-family field. When we want to specify that the initial state is $\o\in\O$, we write $\h^\o$ for the corresponding process. When the PCA has no memory (i.e.\ for all $\cU\in\supp\u$ and $U\in \cU$ we have $(U\cap[-r,r]^d\times \{-1\})\in\cU$), we may abusively write $\h^\o$ with $\o\in\{0,1\}^{\bbZ^d}$ and it is understood that this is the state of the process at time $-1$. We write $\bbP_\u$ for the law of the $\cU_{x,t}$ field, from which the process is constructed.

\subsubsection{Examples}
\label{subsec:examples}
Let us now introduce a few relevant examples.

\paragraph{Toom rule with death} The Toom rule \cite{Toom80}*{Example 1} is a deterministic CA in two dimensions, which updates the state of each site $x\in\bbZ^2$ to the more common value (in $\{0,1\}$) among the current state of $x$, $x+(1,0)$ and $x+(0,1)$. We further subject this CA to a specific type of noise, obtaining a PCA that we will refer to as \emph{Toom rule with death}. Namely, at each step and each site independently with probability $1-p\in[0,1]$, instead of applying the previous rule, we directly set it to state $0$. 

With our notation this corresponds to $d=2$, $r=1$ and $\u$ charging only two up-families: $\u(\{\varnothing\})=1-p$ (we will systematically put accolades to avoid confusing e.g.\ the singleton system consisting of the empty up-family appearing above and the empty system, which naturally verifies $\u(\varnothing)=0$) and
\[\u\left(\left\{\left\{X\subset R:|X\cap\{(0,0,-1),(1,0,-1),(0,1,-1)\}|\ge 2\right\}\right\}\right)=p.\]

This PCA is attractive, but not additive. For $p=1$ it degenerates into the (deterministic) CA called Toom rule. In fact, Toom \cite{Toom80} studied random perturbations of attractive CA in much greater generality, but we will come back to this later.

\paragraph{Generalised oriented site percolation} Fix $X\subset R$ and $p\in[0,1]$. We define GOSP to be the additive attractive PCA with neighbourhood $X$, given by $\u(\{\varnothing\})=1-p$ and $\u(\{\{Y\subset R:Y\cap X\neq\varnothing\}\})=p$. The name comes from the observation that $\h^{\{0\}}(t)\neq\bzero$ if and only if there is a path with steps in $-X$ from $0$ to some site of the form $(x,t)\in\bbZ^{d+1}$, using only sites $(y,s)\in\bbZ^{d+1}$ such that $\cU_{y,s}\neq \varnothing$. The standard oriented site percolation model is recovered by taking $X=\{(-1,-1),(1,-1)\}$ in one dimension.

\paragraph{Bootstrap percolation} A BP model is specified by an \emph{update family}: a finite family $\cX$ of finite subsets of $\bbZ^d\setminus\{0\}$, both the sets and the family being possibly empty. At each time step the state of a site $x\in\bbZ^d$ becomes $1$ if it is already $1$ or there exists $X\in\cX$ such that all elements of $x+X$ are in state $1$. In fact, this is just another way to parametrise the set of all attractive CA with no memory which are monotone in time in the sense that $\h_x(t+1)\ge \h_x(t)$ for all $x\in\bbZ^d$ and $t\ge 0$.\footnote{This property is sometimes called \emph{freezing} to distinguish from attractiveness, which is also a type of monotonicity.}

With our notation BP corresponds to taking as $\u$ the Dirac measure on the minimal up-family $\cU\subset \O_R$ such that $\{X\times\{-1\}:X\in(\cX\cup\{\{0\}\})\}\subset \cU$. If we drop the assumption that $\{0\}\times\{-1\}\in \cU$, we retrieve the class of all attractive CA with no memory. The latter is essentially the class of models whose random perturbations were studied by Toom \cite{Toom80}.

More generally, we define \emph{inhomogeneous BP} by a measure $\chi$ on update families. Then each site $x\in\bbZ^d$ is assigned an i.i.d.\ update family $\cX_x$ with law $\chi$ and at each step at site $x$ we use the minimal up-family $\cU_x\subset \O_R$ such that $\{X\times \{-1\}:X\in(\cX_x\cup\{\{0\}\})\}\subset\cU_x$ and define the evolution via \cref{eq:def:PCA}\footnote{There may be issues defining this if $\chi$ has infinite support. We will only consider finitely supported $\chi$ measures in this work.} with $\cU_{x,t}=\cU_x$ for all $t$. Clearly, this is no longer a CA or a PCA, but rather what one would call \emph{an inhomogeneous attractive CA}.

\paragraph{PCA with death}
Given an attractive PCA, we define its version with death by considering $\tu=p\u+(1-p)\d_{\varnothing}$, so that $\tu$ defines another attractive PCA. In words, we run the original PCA with probability $p$ and put state $0$ with probability $1-p$, like we did for the Toom rule with death and GOSP.

\paragraph{Kinetically constrained models} 
KCM are continuous time Markov processes with state space $\{0,1\}^{\bbZ^d}$ informally defined as follows, given an \emph{update family}: a finite family $\cX$ of finite subsets of $\bbZ^d\setminus\{0\}$, and a parameter $q\in[0,1]$ (see \cite{Cancrini08}). Each site $x\in\bbZ^d$ attempts to update at rate $1$ to an independent Bernoulli state with parameter $q$, but is only allowed to do so if the configuration at some $X\in\cX$ is $\bone$ (the all $1$ configuration). Otherwise, the state of $x$ cannot change until the above constraint becomes satisfied.

Superficially, KCM are not closely related to the attractive PCA that we study, as they are neither attractive, nor discrete time, nor synchronous. Nevertheless, we will see that our treatment entails new results for KCM as well.

\subsection{The phase diagram of PCA}
\label{subsec:phases}
In 1994 Bezuidenhout and Gray \cite{Bezuidenhout94} established the following fundamental result (see their work for a more formal statement).
\begin{thm}
\label{th:BG}
Within the set of attractive PCA $\u$ without memory, the set of those with $\u(\{\varnothing\})>0$ and $\bbP_\u\left(\forall t>0,\h^{\{0\}}(t)\neq\bzero\right)>0$ is open.
\end{thm}
The main corollary of this is that the phase transition of survival from a single point for attractive PCA without memory but with positive death rate is continuous. Moreover, they showed that, within this `supercritical' phase described in the theorem, one can perform renormalisation to highly supercritical oriented percolation. This entails a number of results and is a key step towards establishing that models in this phase are `well-behaved' (see \cite{Hartarsky21GOSP} for more detail on what we mean by this).

Our first main result is of a similar flavor and complements \cref{th:BG}.
\begin{thm}
\label{th:main}
Let $S$ be in the set of attractive PCA such that $\d_\bzero$ is their unique invariant measure. Let $\Int$ and $\overline{\cdot}$ be the interior and the closure within the set of absorbing attractive PCA. Then for any $\u\in \Int(S)$ there exist $c,C>0$ such that for all $t>0$ and finite $A\subset \bbZ^d\times\{-r,\dots,-1\}$ it holds that
\begin{align}
\label{eq:main:1}\bbP_\u\left(\h_0^{\bone}(t)\neq 0\right)&{}\le Ce^{-ct},\\
\label{eq:main:2}\bbP_\u\left(\h^A(t)\neq\bzero\right)&{}\le C|A|e^{-ct}.
\end{align}
Moreover, $S\subset \overline{\Int (S)}$.
\end{thm}
This result can be informally rephrased as `the subcritical phase of attractive PCA is well-behaved.' We leave the discussion of how to use this on concrete models to \cref{subsec:parametrisation}.

Unfortunately, the terms `supercritical' for \cref{th:BG} and `subcritical' for \cref{th:main} are quite misleading. Indeed, it is known that there is an intermediate regime, which may naturally be called \emph{cooperative survival} phase. More specifically, the Toom rule with death for small enough death rates exhibits such behaviour (see \cites{Toom80}) and more general examples will be discussed in \cref{subsec:BP}.

\subsection{Parametrised models and applications}
\label{subsec:parametrisation}
In practice one is usually not interested in the set of \emph{all} PCA, but rather has one specific PCA in mind, possibly with a parameter to tune. As it stands, \cref{th:main} does not assert that for a specific model, as we vary the parameter, we will have a well-behaved subcritical phase immediately followed by a non-uniqueness one. Our next goal is to provide a simple sufficient criterion for this to happen. Indeed, hypotheses are needed, since the model of interest may glide along the boundary of the subcritical phase, in which case no sharpness of the phase transition is to be expected. To preclude this scenario we will require a few more definitions.

We say that a measure $\n$ on a partially ordered set $(X,\ge)$ \emph{stochastically dominates} another one, $\u$, if for every down-set $D$ of $X$ we have $\n(D)\le\u(D)$. A \emph{parametrised model} is a continuously differentiable curve $(\u_p)_{p\in[p_1,p_2]}$ in the space of attractive PCA for some $p_1<p_2\in\bbR$, so that $\u_p$ are probability measures on $\sU$. Here and below derivatives $\u'_p=\lim_{q\to p}(u'_q-u'_p)/(q-p)$ are viewed as signed measures on $\sU$ equipped with the weak topology. Note that the function $p\mapsto \u_p$ is nondecreasing for stochastic domination (i.e.\ for all $p\le q$ the measure $\u_q$ stochastically dominates $\u_p$) iff 
\begin{equation}
\label{eq:def:monotonicity}\forall\sD\in\fD\setminus\{\varnothing,\sU\},\forall p\in[p_1,p_2],\u'_p(\sD)\le 0.
\end{equation}
If that is the case, we say that the parametrised model is \emph{nondecreasing} and, if additionally $\u_p\neq\u_q$ whenever $p\neq q$, we say it is \emph{increasing}. 

Morally, any increasing parametrised model should have a sharp phase transition in the sense of \cref{th:main} (see \cref{th:parametrised} below). However, even for percolation some issues may arise if the model increases in a `nonessential' way. Furthermore, there has been some trouble establishing a necessary and sufficient condition for a modification being `essential' \cites{Balister14,Aizenman91}. Circumventing these issues, we propose a sufficient criterion, which is satisfied in natural models.

We say the parametrised model is \emph{strongly increasing} if there exists $c>0$ such that
\begin{equation}
    \label{eq:def:strong:monotonicity}
\forall \sD\in \fD\setminus\{\varnothing,\sU\}, \forall p\in[p_1,p_2],\u_p'(\sD)\le -c(1-\u_p(\sD)),
\end{equation}
where the derivative is w.r.t.\ $p$. For a nondecreasing parametrised model we define the \emph{critical parameter} \begin{equation}
\label{eq:def:pc}
\pc=\inf\left\{p\in[p_1,p_2]: \d_\bzero\text{ is not the unique invariant measure}\right\}
\end{equation}
with $\inf\varnothing=p_2$. Note that $\u_p(\{\varnothing\})=0$ iff $\d_\bone$ is invariant, while $\u_p(\{\O_R\})=0$ iff $\d_\bzero$ is invariant. We will say that a nondecreasing parametrised model has a \emph{sharp transition} if for any $p\in[p_1,\pc)$ there exist $c,C>0$ such that \cref{eq:main:1,eq:main:2} hold for all $t>0$ and finite $A\subset\bbZ^{d}\times\{-r,\dots,-1\}$.
\begin{thm}
\label{th:parametrised}
Every strongly increasing parametrised model has a sharp transition.
\end{thm}
\begin{rem}
\Cref{th:main,th:parametrised} and their proofs extend naturally beyond the binary setting. More precisely, one may replace the base set $\{0,1\}$ of $\O$ by an arbitrary finite partially ordered set with unique minimal and maximal elements called $0$ and $1$ (the maximal one is not essential). We have chosen to reason directly in the binary case for the sake of readability.
\end{rem}

We next state a few interesting applications of \cref{th:parametrised} to specific models. The first one recovers a classical result of Aizenman and Barsky \cite{Aizenman87}*{Theorem 7.3} (also see their section 8.1 for a discussion of additive PCA) and Menshikov \cite{Menshikov86}, which also has other proofs \cites{Duminil-Copin16,Duminil-Copin19}.
\begin{cor}
\label{cor:GOSP}
Every GOSP has a sharp transition.
\end{cor}
Indeed, by the definition in \cref{subsec:examples} of GOSP with neighbourhood $X\subset R$, for all $p\in(0,1]$ (the case $p=0$ is similar) and $\sD\in\fD\setminus\{\varnothing\}$ we have
\[\u'(\sD)=\u'(\{\varnothing\})+\1_{\cU\in\sD}\u'(\{\cU\})=-1+\1_{\cU\in\sD}=\frac{-1+\u(\sD)}{p}\le-1+\u(\sD),\]
where we set $\cU=\{Y\subset R:Y\cap X\neq\varnothing\}$. However, we may also directly use \cref{th:linear} below instead of \cref{th:parametrised} to deduce \cref{cor:GOSP}.

We next state a consequence of \cref{th:parametrised} in the BP context.
\begin{thm}
\label{th:BP}
Consider a BP with update family $\cX$ contained in a half-space: there exists $u\in\bbR^d$ such that $\forall X\in\cX,\forall x\in X,\<x,u\>< 0$. We use i.i.d.\ Bernoulli initial condition with parameter $q$, whose law we denote by $\m_q$ (BP being a CA, this is the only randomness). Then, setting 
\begin{equation}
    \label{eq:def:qc}
\qc=\inf\left\{q\in[0,1]:\lim_{t\to\infty}\m_q(\h_0(t)\neq 1)=0\right\}.\end{equation}
we have
\begin{equation}
\label{eq:BP:exp:decay}
\forall q>\qc,\exists c,C>0,\forall t>0,\quad\m_q\left(\h_0(t)\neq 1\right)\le Ce^{-ct}.
\end{equation}
\end{thm}
\Cref{th:BP} makes further progress towards proving \cite{Hartarsky21}*{Conjecture 8.1}, which asks for this result for all families, not necessarily contained in a half-space. This conjecture itself generalised a question of Schonmann \cite{Schonmann92} from 1992, which remains open, asking if the same result holds when restricted to $\cX$ contained in the set of nearest neighbours of the origin instead of a half-space.

It should be noted that \cref{th:BP} is not a direct application of \cref{th:parametrised}, since BP has $\u(\{\varnothing\})=0$ by definition. Instead we will rely on the correspondence between BP and CA of \cref{prop:correspondence}. This correspondence presented in \cref{sec:BP} is also at the root of the results discussed in \cref{subsec:BP}. To give a simple instance of it, let us focus on standard oriented site percolation from \cref{subsec:examples} viewed as a one-dimensional CA with death. Consider the two-diemsnional BP update family $\cX=\{\{(-1,-1),(1,-1)\}\}$ and the initial condition given by the space-time sites whose (random) up-set is $\varnothing$ (death). Then the space-time sites $(x,t)$ which remain in state 0 until time $t$ in the BP process are exactly the ones which are in state 1 for the oriented percolation CA with death with initial condition $\bone$. 

\cref{th:BP} generalises directly to inhomogeneous BP as follows.
\begin{thm}
\label{th:inhomogeneous}
Consider an inhomogeneous BP with measure $\chi$ supported on a finite set of update families contained in the same half-space (see \cref{th:BP}). Denote by $\m_q$ the law of the i.i.d.\ update families with distribution $\chi$ and i.i.d.\ initial condition with Bernoulli law of parameter $q$. Then \cref{eq:BP:exp:decay} holds with $\qc$ defined by \cref{eq:def:qc}.
\end{thm}
 
Moving on to KCM, the following is a direct consequence of \cref{th:BP} together with \cite{Hartarsky21}*{Theorem~3.7}.
\begin{cor}
\label{cor:KCM}
Consider a KCM with update family $\cX$ contained in a half space (see \cref{th:BP}). Then the spectral gap of its generator is positive for all $q>\qc$ and $0$ for all $q<\qc$, where $\qc$ is the quantity defined in \cref{eq:def:qc} for the BP with the same choice of update family $\cX$.
\end{cor}

\subsection{Bootstrap percolation and Toom perturbations of attractive cellular automata}
\label{subsec:BP}
Finally, we discuss consequences of the PCA representation of BP used to prove \cref{th:BP}.

\paragraph{Bootstrap percolation universality}
We first need to introduce a few notions from BP universality, whose aim is to classify BP update families according to their behaviour at or around their phase transition.

We say that a BP update family $\cX$ is \emph{supercritical} if there exists a finite set $A\subset\bbZ^d$ such that $\bigcup_{t>0}\h^A(t)$ is infinite. A more geometric characterisation is available based on the notion of stable direction. We say that $u\in S^{d-1}$ (the unit sphere) is \emph{unstable} if there exists $X\in\cX$ such that for all $x\in X$ we have $\<u,x\><0$ and it is \emph{stable otherwise}. It was proved in \cite{Bollobas15}*{Theorem 7.1, Definition 1.3} that in two dimensions $\cX$ is supercritical if and only if there exists an open hemisphere of unstable directions. This is expected to generalise to any dimension.

Instead, we say that $\cX$ is \emph{subcritical} if $\qc>0$ (recall \cref{eq:def:qc}). It was proved in \cite{Balister22}*{Definition 1.2, Corollary 1.6} that $\cX$ is subcritical if and only if every open hemisphere contains an open set of stable directions.

\paragraph{Toom perturbations}
Toom \cite{Toom80} considered random space-time perturbations of attractive CA with $\u(\{\varnothing\})=0$ (otherwise every configuration becomes $\bzero$ after one step, since CA are deterministic), in order to assess whether the $\bzero$ and $\bone$ states are stable. We will not define the exact noise used there, but using attractiveness to suppress `positive' noise, it can be brought down to a noise stochastically dominated by a product measure with low parameter. For simplicity, we will work directly with the resulting product noise. Namely, we consider the CA with death $\tu$. We say that the CA $\u$ is an \emph{eroder} if, starting from any configuration $\o$ such that $\o_{(x,s)}=0$ for finitely many $(x,s)\in\bbZ^{d}\times \{-r,\dots,-1\}$, we have $\h^\o(t)=\bone$ for any $t$ large enough depending on $\o$. 

\paragraph{Interplay}
We can recover the following result, first established by Toom \cite{Toom80}, directly from the BP--PCA correspondence and BP results.
\begin{thm}
\label{th:Toom}
In dimension $d=1$ any attractive CA is an eroder iff its version with death rate $1-p$ small enough has an invariant measure different from $\d_\bzero$.
\end{thm}
Let us note that in our alternative proof of \cref{th:Toom} we will not at all require the full power of \cite{Balister22}*{Corollary 1.6} or its two-dimensional version \cite{Balister16}*{Theorem 1}. Instead, we only rely on the much easier partial result \cite{Balister16}*{Theorem 9} in view of \cref{lem:directions} below. We believe the renormalisation approach of \cite{Balister16}*{Theorem 9} to be simpler than the original proof of Toom, as well as its subsequent versions in \cites{Gacs21,Berman88,Swart22, Hartarsky22Toom} based on somewhat involved Peierls arguments. Naturally, a carefully performed Peierls argument often gives a better quantitative bound on the critical parameter (see e.g.\ \cites{Swart22,Hartarsky22Toom}), though this is seldom important. We also mention the more complex renormalisation approach of \cite{Bramson91}.

Inversely, the viewpoint of \cite{Toom80}*{Theorem 5} provides an interesting consequence for BP, thanks to the BP-PCA correspondence of \cref{prop:correspondence}. Namely, it allows us to recover the main results of \cites{Balister16,Bollobas15} restricted to families contained in half-spaces, but extended to arbitrary dimension, which have not been treated until present (but see the subsequent work \cite{Balister22}).
\begin{thm}
\label{th:orientation:BP}
A BP in any dimension whose update family is contained in a half-space (recall \cref{th:BP}) is either supercritical or subcritical (that is, either there are finite sets with infinite offspring or $\qc>0$).
\end{thm}
Although in two dimensions this can be checked easily (see \cref{lem:directions} below) from geometric considerations based on the characterisations of supercritical and subcritical models from \cites{Bollobas15,Balister16}, it does not seem to have been noticed. We should note that the half-space condition cannot be removed, as without it already in two dimensions another, \emph{critical}, class emerges \cite{Bollobas15}.
\subsection{Organisation}
The remainder of the paper is organised as follows. In \cref{sec:decay} we prove our general results---\cref{th:main,th:parametrised} by adapting the randomised algorithm approach of Duminil-Copin, Raoufi and Tassion \cite{Duminil-Copin19}. In \cref{sec:BP} we introduce the correspondence between BP and CA and deduce \cref{th:BP,th:inhomogeneous,th:Toom,th:orientation:BP}.

\section{Sharpness of the transition}
\label{sec:decay}
\subsection{Linear parametrisation}
\label{subsec:linear}
We first seek to prove the following preliminary result, from which \cref{th:main,th:parametrised} will be deduced in \cref{subsec:parametrised}.
\begin{thm}
\label{th:linear}
Let $\u$ be a rates measure and let $\u_p=p\u+(1-p)\d_{\varnothing}$ for $p\in[0,1/(1-\u(\{\varnothing\}))]$. This nondecreasing parametrised model has a sharp transition.
\end{thm}
The first step is a Russo formula adapted to our situation. We say that a space-time site $(x_0,t_0)$ is \emph{pivotal} for an event $A$ and realisation of the $\cU_{x,t}$ variables, if $A$ occurs for $(\cU_{x,t})_{(x,t)\in\bbZ^{d+1}}$, but it does not occur if we replace $\cU_{x_0,t_0}$ by $\varnothing$.
\begin{lem}
\label{prop:russo}
Let $\u$ be a rates measure and let $\u_p=p\u+(1-p)\d_{\varnothing}$ for $p\in[0,1]$. Let $A_n$ be the event that $\h_0^{\bone}(n)\neq 0$ and set $\theta_n(p)=\bbP_{\u_p}(A_n)$. Then for all $n\ge 0$ and $p\in(0,1]$
\[\theta_n'(p)=\sum_{x,t}\bbP_{\u_p}\left((x,t)\text{ is pivotal for }A_n\right).\]
\end{lem}
\begin{proof}
Fix an i.i.d.\ random field $\cU_{x,t}^1$ with law $\u$ and i.i.d.\ random variables $X_{x,t}$ uniform on $[0,1]$ and define
\[\cU_{x,t}^p=\begin{cases}\cU_{x,t}^1&\text{if }X_{x,t}\le p\\
\varnothing&\text{if }X_{x,t}> p.\end{cases}\]
Conditioning on the $\cU_{x,t}^1$, we can then proceed as in the proof of the standard Russo formula (see e.g.\ \cite{Grimmett99}*{Sec. 2.4}). We then average over $\cU_{x,t}^1$ to obtain the desired equality.
\end{proof}

The next step is to adapt the Duminil-Copin--Raoufi--Tassion version of the O'Donnell--Saks--Schramm--Servedio decision tree result (see \cites{ODonnell05,Duminil-Copin19} for background). It is important to note that the support of $\u$ is not necessarily binary or even totally ordered, so it is not clear how to adapt \cite{Duminil-Copin19}*{Theorem 1.1}. We circumvent this problem by restricting our attention to events rather than real-valued observables.
\begin{lem}
\label{lem:DCRT:var}
Let $n\in\bbN$ and $A\subset \sU^n$ be an increasing event (w.r.t.\ the pointwise inclusion order). Fix a randomised algorithm determining the occurrence of $A$ and denote by $\d_i$ the probability that it reveals the value of the $i$-th up-family. Then 
\[\var_\u(\1_A)\le 2\sum_{i=1}^{n}\d_i\cdot\bbP_\u\left(i\text{ is pivotal for }A\right)\]
where $\var$ and $\bbP$ are w.r.t.\ a product measure $\bigotimes_{i=1}^n\u$ on $\sU^n$.
\end{lem}
\begin{proof}
The result is proved like Theorem 1.1.\ of \cite{Duminil-Copin19} (also see Remark 2.2.\ there). It suffices to replace their equation (5) by the fact that if $\h\in\sU^n$ and $\o\in\sU^n$ differ only at the $i$-th up-family, then
\begin{equation}
\label{eq:DCRT:piv}
\left|\1_A(\h)-\1_A(\o)\right|\le \1_{i\text{ is pivotal for $A$}}(\h)+\1_{i\text{ is pivotal for $A$}}(\o).
\end{equation}
To see this, assume w.l.o.g.\ that $\h\not\in A$ (if $\1_A(\h)=\1_A(\o)$ there is nothing to prove). Then replacing $\h_i$ by $\varnothing$ does not trigger the occurrence of $A$, as $A$ is increasing. Hence, either $\o\not\in A$ and the l.h.s.\ in \cref{eq:DCRT:piv} is 0, or $i$ is pivotal for $A$ and $\o$, so the r.h.s.\ is 1.
\end{proof}

Our next task is to define a suitable randomised algorithm to which \cref{lem:DCRT:var} will be applied. 

\begin{lem}
\label{lem:algo}
Fix a parametrised model and $p\in[p_1,p_2]$. Recall $A_n=\{\h_0^{\bone}(n)\neq 0\}$ and $\theta_n(p)=\bbP_p(A_n)$. For each integer $n\ge 1$ there exists a randomised algorithm determining the occurrence of $A_n$ such that for every $(x,t)\in\bbZ^{d+1}$ its revealment probability $\d_{x,t}$ satisfies
\[\d_{x,t}\le \frac{2}{n}\sum_{i=0}^{n-1}\theta_i(p).\]
\end{lem}
\begin{proof}
Clearly, it suffices to restrict our attention to the up-families in 
\begin{equation}
    \label{eq:def:S}
    S=\left\{(x,t)\in \bbZ^{d+1}:|x|\le r\cdot (n-t),t\in[0,n]\right\},
\end{equation}
where $r$ is the range of the process. The algorithm proceeds as follows (following \cite{Duminil-Copin19}*{Lemma 3.2} and \cite{Hartarsky21}*{Lemma 7.3}). Select a number $k$ uniformly at random in $\{1,\dots,n\}$. Contrary to \cite{Hartarsky21}, we will explore forward in time, starting from time $k$. For $t\in\{0,\dots,n-k\}$ denote by $\o^\bone(t)$ the configuration with initial condition $\bone$ and up-families $\cU'_{x,t}=\cU_{x,t+k}$, where $\cU_{x,k}$ are the up-families in the graphical construction of $\h$. We initialise the algorithm at $t=0$. We explore the state of each $\cU_{x,t+k}$ for $(x,t+k)\in S$ such that the currently explored up-families do not allow to conclude that $\o^\bone_x(t)=0$ (so, at the first step that is all $\cU_{x,k}$ for $(x,k)\in S$). Then we increment $t$ and repeat the previous operation. Upon reaching $t=n-k$, if we are able to conclude that $\o^{\bone}_0(n-k)=0$, we terminate the algorithm (since we know by attractiveness that $\h^\bone_0(n)\le \o^{\bone}_0(n-k)=0$). Otherwise we reveal all up-families in $S$ to determine if $A_n$ occurs and terminate.

Next observe that, conditionally on $k$, the probability of revealing $U_{x,t}$ for $(x,t)\in S$ is at most $\theta_{n-k}(p)+\1_{t\ge k}\theta_{t-k}(p)$, the first term bounding the probability that we reach $t=n-k$ without having $\o_0^{\bone}(n-k)=0$. Averaging over the law of $k$, we obtain the desired conclusion.
\end{proof} 
We are now ready to assemble the proof of \cref{th:linear}.
\begin{proof}[Proof of \cref{th:linear}]
Fix $\u\neq\d_\varnothing$ and $\pc>0$ (otherwise the statement is trivial). Up to linear reparametrisation of the model, we may further assume that $\u(\{\varnothing\})=0$, $p_1=0$ and $p_2=1$. Let $\u_p=p\u+(1-p)\d_{\varnothing}$ for $p\in[0,1]$. This clearly defines a nondecreasing parametrised model, since $\cU\supset\varnothing$ for all $\cU\in\sU$. Define $A_n$, $\theta_n(p)$ and $S$ as above (see \cref{lem:algo} and \cref{eq:def:S}). By \cref{lem:DCRT:var,lem:algo}, for all $p\in[0,1]$ we have
\[\theta_n(p)(1-\theta_n(p))\le \frac{4}{n}\bbE_{\u_p}[N]\sum_{i=0}^{n-1}\theta_{i}(p),\]
where $N$ is the number of $(x,t)\in S$ which are pivotal for $A_n$. Then, applying \cref{prop:russo} we get that for $p\in(0,1]$
\[\theta_n'(p)\ge \frac{n\theta_n(p)(1-\theta_n(p))}{4\sum_{i=0}^{n-1}\theta_{i}(p)}.\]

Fix $\varepsilon \in(0,\pc)$. Observe that for any $p\in[\varepsilon,\pc-\varepsilon]$ we have $1-\theta_n(p)\ge \u_p(\{\varnothing\})\ge \varepsilon$. Thus, we may apply \cite{Duminil-Copin19}*{Lemma 3.1} in this interval to get that either \cref{eq:main:1} holds for $\u_p$ for all $p\in[\e,\pc-\varepsilon)$, or there exists $p\in(\varepsilon,\pc-\varepsilon)$ such that $\theta(p)>0$. However, the latter contradicts the definition of $\pc$, \cref{eq:def:pc}. Since \cref{eq:main:1} is trivial for $\u_0=\d_\varnothing$, we have established \cref{eq:main:1} for $\u_p$ for all $p\in[0,\pc)$.

Finally, it remains to derive \cref{eq:main:2}. Starting from a finite set $A$, the only sites which may be in state $1$ at time $t$ are those at distance at most $rt$ from a site in $A$. Hence, taking the union bound of \cref{eq:main:1} over these sites and using attractiveness, we get that for all $p\in[0,\pc)$
\[\bbP_{\u_p}\left(\h^A(t)\neq \bzero\right)\le |A|rt\max_{x\in\bbZ^d}\bbP_{\u_p}\left(\h^A_{x}(t)\neq0\right)\le C|A|rte^{-ct}\le C'|A|e^{-c't}\]
for a suitable choice of $C',c'>0$ depending on $p$, yielding \cref{eq:main:2}.
\end{proof}

\subsection{Parametrisations}
\label{subsec:parametrised}

With \cref{th:linear}, we are in position to conclude the proof of \cref{th:main}.
\begin{proof}[Proof of \cref{th:main}]
Let $\u$ be in the interior of the set of absorbing attractive PCA such that $\d_\bzero$ is their unique invariant measure, denoted by $\Int(S)$. Clearly, we cannot have $\u(\{\varnothing\})=0$, since in that case $\d_\bone$ is invariant. Therefore, we may consider the linearly parametrised model $\u_p=p\u+(1-p)\d_\varnothing$ for $p\in[0,1/(1-\u(\{\varnothing\}))]$. Since this parametrised model is absorbing and $\u_1\in \Int(S)$, necessarily $\pc>1$. But then \cref{th:linear} yields \cref{eq:main:1,eq:main:2} for $\u=\u_1$.

Similarly, assume that $\d_\bzero$ is the unique invariant measure of the $\u$-PCA, but not necessarily $\u\in \Int(S)$. Then \cref{th:linear} gives that $\pc\ge 1$, so that \cref{eq:main:1,eq:main:2} hold for any $p<1$. Then \cref{prop:open} below allows us to conclude that $\u_p\in \Int(S)$ for $p\in[0,1)$ and conclude that $\u\in\overline{\Int(S)}$.
\end{proof}
\begin{prop}
\label{prop:open}
The set of attractive absorbing PCA for which there exist $c,C>0$ so that \cref{eq:main:1} holds for all $t>0$ is open.
\end{prop}
\begin{proof}
We use a standard renormalisation argument. Fix $\u$ such that \cref{eq:main:1} holds for some $c,C>0$ and let $\e>0$ small enough to be chosen later. Partition $\bbZ^{d+1}$ into the boxes $B_{x,t}=(xL,tL)+[0,L)^{d+1}$ for some large $L$ to be chosen later. Let $\cB_{x,t}$ denote the event that the $\u$-PCA with initial condition $\bone$ and up-family field translated by $-L(x,t)$ is $1$ for some $(y,s)\in [0,L)^d\times [L-r,L)$. By \cref{eq:main:1} and the union bound we have that for $L$ large enough $\bbP_\u(\cB_{x,t})<\e$ for any $(x,t)\in\bbZ^{d+1}$. Moreover, $\cB_{x,t}$ only depends on the up-families $\cU_{y,s}$ for $(y,s)\in B_{x',t'}$ with $(x',t')$ at distance at most $r$ from $(x,t)$.

Fix an attractive absorbing PCA $\n$ sufficiently close to $\u$ (depending on $\e$ and $L$). Then we can couple the corresponding up-family fields $\cU_{y,s}$ and $\cU'_{y,s}$, so that they differ with small probability and independently for each $(y,s)\in\bbZ^{d+1}$. We say that the box $B_{x,t}$ is \emph{bad} if $\cB_{x,t}$ occurs or there exists a site $y,s$ at distance at most $rL$ from $B_{x,t}$ such that $\cU_{y,s}\neq\cU'_{y,s}$. By the well-known Liggett--Schonmann--Stacey theorem \cite{Liggett97} the good boxes stochastically dominate a product measure with density of bad boxes approaching $1$ if we choose $\e$ small enough.

Finally, we consider oriented percolation of bad boxes with range $r$ and observe that if a box does not belong to a bad connected component reaching time $0$, then the state at its top boundary is necessarily $\bzero$ in both the $\u$ and the $\n$-PCA with initial condition $\bone$. Indeed, if the box $B_{x,t}$ itself is not bad that is enough by definition of bad, boxes, while otherwise we can proceed by induction. Namely, we observe that if the state of all boxes $B_{y,t-1}$ for $\|x-y\|_\infty\le r$ are good, then the their top boundaries are in state $\bzero$ and no absorbing PCA of range $r$ can reach a nonzero state at the top of $B_{x,t}$, starting from that.

We can then conclude that \cref{eq:main:1} holds for $\n$ (and different $c,C>0$), since in highly subcritical independent percolation of range $r$ the size of clusters has this exponential decay property (e.g.\ by \cref{cor:GOSP}, which was a consequence of \cref{th:linear} without going through \cref{th:parametrised}).
\end{proof}

Our next goal is to prove \cref{th:parametrised}. We will deduce this from the specific case, \cref{th:linear}, and the following result.
\begin{lem}
\label{lem:strong:monotonicity}
Let $(\u_p)_{p\in[0,1]}$ be a strongly increasing parametrised model. Then there exists $c'>0$ such that for all $p\in(0,1]$ there exists $\e>0$ such that for all $p'\in [1-\e,1]$ it holds that $\n_{p'}:=p'\u_{p}+(1-p')\d_{\varnothing}$ stochastically dominates $\u_{p-(1-p')/c'}$.
\end{lem}
\begin{proof}
Notice that for any down-system $\sD\in\fD\setminus\{\varnothing,\sU\}$ the desired inequality $\n_{p'}(\sD)\le \u_{p-(1-p')/c'}(\sD)$ is an equality for $p'=1$. Therefore, it suffices to show that the derivatives satisfy the inverse inequality:
\begin{equation}
\label{eq:ineq:domination}
\u_p(\sD)-1
=\n'_{p'}(\sD)\ge \frac{\partial\u_{p-(1-p')/c'}}{\partial p'}(\sD)=\frac{\u'_{p-(1-p')/c'}(\sD)}{c'}\end{equation}
in the neighbourhood of $p'=1$.

If $\u_p(\sD)=1$, we may conclude directly by the fact that the parametrised model is nondecreasing. Otherwise, recalling \cref{eq:def:strong:monotonicity} and setting $c'=c/2$, we have
\[\frac{\u'_{p-(1-p')/c'}(\sD)}{c'}\le  2\left(\u_{p-(1-p')/c'}(\sD)-1\right)\xrightarrow{p'\to1}2\left(\u_p(\sD)-1\right)<\u_p(\sD)-1.\]
Hence, for $p'$ sufficiently close to $1$, \cref{eq:ineq:domination} does hold for all $\sD\in\fD\setminus\{\varnothing,\sU\}$, since $\fD$ is finite.
\end{proof}
\begin{proof}[Proof of \cref{th:parametrised}]
Consider a strongly increasing parametrised model $(\u_p)_{p\in[p_1,p_2]}$. Up to linear reparametrisation, we may assume that $p_1=0,p_2=1$. We further suppose that $\pc>0$, as otherwise there is nothing to prove. Fix $p\in(0,\pc)$ and set $\n_{p'}=p'\u_{p}+(1-p')\d_{\varnothing}$ for $p'\in [1-\e,1]$, where $\e$ is from \cref{lem:strong:monotonicity}. Let $\pc'$ denote the critical value of this parametrised model. Then \cref{th:linear} shows that \cref{eq:main:1,eq:main:2} hold for $\n_{p'}$ for $p'<\pc'$. By \cref{lem:strong:monotonicity} and the fact that $(\u_p)_{p\in[0,1]}$ is nondecreasing, it suffices to prove that $\pc'=1$. But this is clear, because the definition of $\pc$ gives that $\d_\bzero$ is the unique invariant measure for $\u_p=\n_1$.
\end{proof}

\section{Bootstrap percolation and cellular automata}
\label{sec:BP}
\subsection{The correspondence}
It was noticed already by Schonmann \cite{Schonmann92} that standard oriented site percolation, instead of a one-dimensional PCA or a two-dimensional percolation, can be viewed as a two-dimensional BP with an i.i.d.\ initial condition. This was exploited in \cite{Hartarsky21} and extended to GOSP subsequently studied in \cite{Hartarsky21GOSP}. We now show that this correspondence in fact extends to all attractive PCA.\footnote{We refer the reader to \cite{Lebowitz90} for a related correspondence between PCA and equilibrium statistical mechanics models.}
\begin{prop}[BP--CA correspondence]
\label{prop:correspondence}
Consider an attractive CA with $\u=\d_{\cU}$ for some $\cU\in\sU$. Let $\th$ be its version with death with rates measure $\tu=p\u+(1-p)\d_\varnothing$ for some $p\in[0,1]$. Let $\cU_{x,t}$ be the up-family field used to construct the $\tu$-PCA. Define the $d+1$-dimensional up-family 
\begin{equation}
\label{eq:correspondence}
\cX=\left\{R\setminus D:D\in\O_R\setminus\cU\right\}.
\end{equation}
Consider the BP process $\o$ defined by the update family $\cX$ with initial condition $(\1_{\cU_{x,t}=\varnothing})_{(x,t)\in\bbZ^{d+1}}$. Define its \emph{closure} 
\begin{equation}
\label{eq:def:closure}
C=\left\{(x,t)\in\bbZ^{d+1}:\exists s\in\{0,1,\dots\},\o_{(x,t)}(s)=1\right\}
\end{equation}
and the configuration $C_0=(\bbZ^d\times\{-r,\dots,-1\})\setminus C$. Then for all $t\ge 0$
\begin{equation}
\label{eq:etatilde}
\th^{C_0}(t)=\{x\in\bbZ^d:(x,t)\not\in C\}
\end{equation}
and this is a version of the stationary $\tu$-PCA at its upper invariant measure, that is, the limit as $t\to\infty$ of the law of $\widetilde\h^{\bone}(t)$. In particular, $C=\bbZ^{d+1}$ a.s.\ if and only if $\d_\bzero$ is the unique invariant measure of the $\tu$-PCA.

Moreover, the map $\cU\leftrightarrow\cX$ is a one-to-one correspondence between $d$-dimensional attractive CA and $d+1$-dimensional BP with update family contained in the lower half-space $\bbZ^{d}\times\{-1,-2,\dots\}$.
\end{prop}
\begin{proof}
The bijectiveness follows from the fact that the \emph{double complement} map $\cU\mapsto\{R\setminus D:D\in\O_R\setminus\cU\}$ is an involution of the power set of $\O_R$, since the two complements commute and each of them is an involution.

Fix an attractive CA and its corresponding BP as in the statement. Let us first verify \cref{eq:etatilde}. By induction on $t$, it suffices to verify this for $t=0$.

By definition $(x,0)\in C$ iff $\cU_{x,0}=\varnothing$ (this is the initial condition of $\o$) or there exists a minimal $s>0$ such that $\o_{(x,0)}(s)=1$. The latter, happens iff $R\setminus (\o(s-1)-(x,0))\not\in\cU$ and $\o_{(x,0)}(s-1)=0$. Hence, $(x,0)\in C$ iff $\cU_{x,0}=\varnothing$ or $R\cap (C_0-(x,0))\not\in\cU$. But this is equivalent to $\th_x^{C_0}(0)=0$ by definition, so \cref{eq:etatilde} indeed holds.

Thus, $1-\1_C$ is a trajectory of the $\tu$-PCA. Moreover, its law is clearly invariant by translation in $\bbZ^{d+1}$, so the process $\th^C_0$ is stationary. It remains to verify that this corresponds to the upper invariant measure. To see this it suffices to prove that 
\begin{equation}
\label{eq:omegaeta}
1-\o_{(x,s)}(t)\ge \th^\bone_x(s)\end{equation}
for all $(x,t,s)\in\bbZ^{d}\times\{1,2,\dots\}^2$ such that $s\ge rt$, since $1-\1_C$ is the decreasing limit of $1-\o(t)$. Notice that $1-\o_{(x,s)}(t)$ in fact only depends on the initial condition of $\o$ for $(y,u)\in\bbZ^d\times\{0,\dots,s\}$, since $s\ge rt$. Define the closure $C'$ as in \cref{eq:def:closure} with the process $\o'$ defined like $\o$ but with initial condition $\1_{\cU_{y,s}=\varnothing}$ for $(y,u)\in\bbZ^d\times\{0,\dots,s\}$ and 0 elsewhere. Then in fact $1-\o_{(x,s)}(t)= 1-\o'_{(x,s)}(t)\ge 1-\1_{(x,s)\in C'}=\th^\bone_x(s)$. The last equality is \cref{eq:etatilde} applied to a suitable choice of initial condition (which we may choose freely, since the relation is deterministic and not just a.s.). Thus, \cref{eq:omegaeta} is established and the proof is complete.
\end{proof}

Amusingly, since BP is an attractive CA, one can iterate this correspondence. As an example, consider the $0$-dimensional CA with no memory given by the identity map $0\mapsto 0;1\mapsto1$ (there are four 0-dimensional CA without memory---the identity, the constant $0$, the constant $1$ and a non-attractive one). It is clearly not an eroder, since $0$ does not become $1$. Its corresponding 1-dimensional BP model is East-BP, which makes the right neighbour of a 1 also 1. When applying the correspondence a second time, we obtain the North-East-BP model (up to a linear transformation of the lattice), which is equivalent to standard oriented site percolation.

\begin{prop}
\label{prop:inhomogeneous}
The correspondence of \cref{prop:correspondence} extends to a one-to-one mapping from the set of all $d$-dimensional attractive PCA (death may be integrated directly into $\u$) to $d+1$-dimensional inhomogeneous BP with $\bzero$ initial condition and update families measure $\chi$ supported on the set of update families contained in the lower half-space (deaths corresponding to $\varnothing\in\cX_x$). Namely, $\chi$ is the image of $\u$ via the mapping $\cU\mapsto\cX$ of \cref{eq:correspondence} and vice versa. Then the same conclusions still hold (\cref{eq:etatilde} and it being a stationary $\u$-PCA at its upper invariant measure). 
\end{prop}
The proof is identical to the one of \cref{prop:correspondence} and therefore omitted.

\subsection{Proof of Theorems \ref{th:BP} and \ref{th:inhomogeneous}}
Fix a BP $\o$ with update family $\cX$ contained in a half-space. Up to an invertible linear transformation of the lattice, we may assume that this is the lower half-space and denote by $\u=\d_{\cU}$ the corresponding attractive CA via \cref{prop:correspondence}. The proposition gives us that $\qc$ in \cref{eq:def:qc} for $\cX$ is in fact equal to $1-\pc$ with $\pc$ from \cref{eq:def:pc} for the $\tu$-PCA with death rate $1-p$. \Cref{th:linear} applies to the $\tu$-PCA, so \cref{eq:main:1} holds for $p<\pc$. 

Applying \cref{eq:etatilde} to the initial condition equal to $\1_{\cU_{x,t}=\varnothing}$ for $t\ge 0$ and $0$ otherwise and denoting $C'$ the corresponding closure, we get that $\th_x^\bone(t)=1$ iff $(x,t)\not\in C'$. Hence, for $p<\pc$ there are $c,C>0$ such that for all $(x,t)\in\bbZ^{d+1}$
\[\bbP_{\tu}\left((x,t)\not\in C'\right)=\bbP_{\tu}\left(\th^\bone_x(t)\neq 0\right)\leq Ce^{-ct}.\]
Denote by $\o'$ the $\cX$-BP with the above initial condition. Then, $(x,t)\not\in C'$ iff $\o'_{(x,t)}(t)=0$, since, by induction on $t$, the $\o'$ process becomes stationary in $\bbZ^d\times\{0,\dots,t\}$ after $t$ steps. Finally, since $\o'_{(x,t)}(t)\le\o_{(x,t)}(t)$ by attractiveness, the proof of \cref{th:BP} is concluded. \Cref{th:inhomogeneous} is proved identically in view of \cref{prop:inhomogeneous}.

\subsection{Proof of Theorems \ref{th:Toom} and \ref{th:orientation:BP}}
Fix an attractive CA $\u=\d_\cU$ and let $\cX$ be its corresponding BP update family via \cref{prop:correspondence}. Consider a configuration $\x$ with finitely many $0$s. Applying \cref{eq:etatilde} to this initial condition and its closure $C$, we get that if $\cU$ is not an eroder, then $\cX$ is supercritical. Inversely, if $\cX$ is supercritical, considering the intersection of the closure of a finite initial set with infinite offspring and a horizontal strip of width $r$ (which is clearly finite), we similarly obtain from \cref{eq:etatilde} that $\cU$ is not an eroder.

Furthermore, \cref{prop:correspondence} grants that $\cX$ is subcritical iff $\cU$ has at least two invariant measures when the death rate $1-p$ is small enough. Thus, \cref{th:Toom} is equivalent to \cref{th:orientation:BP} restricted to two-dimensions and, inversely \cref{th:Toom} without restrictions on the dimension is equivalent to \cref{th:orientation:BP}.

Hence, \cref{th:orientation:BP} follows from \cref{th:Toom}, which is valid in all dimensions \cite{Toom80}*{Theorem 5}. Turning to \cref{th:Toom} in one dimension, in order to avoid a circular reasoning, we recall from \cref{subsec:BP} that by \cites{Bollobas15,Balister16} it suffices to verify the following fact, which we prove for completeness (see \cite{Toom80}*{Theorem 6} or \cite{Toom95}*{Theorem 3}).
\begin{lem}
\label{lem:directions}
Fix a $d$-dimensional update family $\cX$ contained in a half-space. Then for any $u\in S^{d-1}$ the set of stable directions $v\in S^{d-1}$ such that $\<u,v\><0\}$ is either empty or has a nonempty interior. Moreover, if $d=2$ and there is no open semicircle of unstable directions, then there exist two opposite directions in the interior of the set of stable directions.
\end{lem}
\begin{proof}
Let $e_1,\dots, e_{d}$ denote the canonical basis of $\bbR^{d+1}$. For concreteness let us assume that the family is contained in the upper half-space $\{x\in\bbR^{d}:\<x,e_{d}\>>0\}$. It is not hard to see from the definition (see \cite{Bollobas15}*{Remark 3.3}) that the set of stable directions can be written as a finite intersection of finite unions of closed hemispheres containing the direction $e_{d}$ in their interior. 

Clearly, any closed hemisphere contains the geodesic between any of its points and any point in its interior. Therefore, for any stable direction the geodesics to a neighbourhood of $e_d$ consist of stable directions. Hence, the set of stable directions is the closure of its own interior. The general conclusion then follows immediately.

In two dimensions, matters are simpler. We already established that stable directions form a connected closed set, that is, a closed interval of $S^1$. Depending on whether it is smaller or larger than a semicircle, this gives the desired conclusion.
\end{proof}

To conclude, let us mention that, given opposite stable directions provided in \cref{lem:directions}, the renormalisation argument sketched in \cite{Balister16} proceeds as follows. Divide the plane into large rhombi in the usual way, so that their sides are close to being perpendicular to these directions. Then perform renormalisation, saying that the rhomubs is good if it is initially in state $\bzero$ for the $\cX$-BP, which happens with high probability if the parameter $q$ is small. Then it suffices to verify that an infinite oriented path of good rhombi will remain in state $0$ forever, which follows from the suitable choice of their geometry.

\section{Further directions}
To conclude, let us mention a few directions for further work. 

Firstly, in view of \cref{th:BG,th:main}, one would naturally like to know what happens in the cooperative survival phase. That is the interior of the set of absorbing attractive PCA $\u$ such that $\bbP_\u(\forall t>0,\h^{\{0\}}(t)\neq\bzero)=0$ and with more than one invariant measure. We are aware of no results in this direction. For instance one may expect the following to be true.
\begin{ques}
For an attractive PCA $\u$ in the cooperative survival phase does one have exponential convergence to the upper invariant measure starting from the $\bone$ initial condition? Equivalently, in the corresponding inhomogeneous BP, does the truncated infection time have an exponential tail, that is, setting $\t_0=\inf\{t> 0:\h_0(t)=1\}$, do we have \[\limsup_{t\to\infty}\frac{\log \bbP(\t_0>t|\t_0<\infty)}{t}<0?\]
\end{ques}
For models, such as the Toom rule with death, for which one can prove that they are in the cooperative survival phase, this follows from the corresponding proof, but we rather ask for non-perturbative results valid throughout the phase.

It would also be very interesting to obtain an analogue of \cref{th:main} for PCA with a unique invariant measure not necessarily equal to $\d_\bzero$. See \cites{Taati21,Louis04} for progress in this direction under other conditions.

Furthermore, it is natural to seek to extend \cref{th:main,th:linear} to continuous time absorbing attractive interacting particle systems with single spin flips. In the case of the contact process this is a well-known result of Bezuidenhout and Grimmett \cite{Bezuidenhout91} (also see \cite{Swart18}).

Finally, in the light of the BP--PCA correspondence of \cref{prop:correspondence,prop:inhomogeneous}, can one transfer more interesting information between the two settings?

\section*{Acknowledgements}
This work is supported by ERC Starting Grant 680275 ``MALIG.'' We thank Ir\`ene Marcovici, Jan Swart, R\'eka Szab\'o, Siamak Taati and Cristina Toninelli for enlightening discussions. We also thank Réka for bringing important references to our attention. We thank the annonymous referees for helpful remarks on the presentation.

\paragraph{Subsequent developments}
Since the submission of this manuscript several particularly related works have been completed and call for comment.

Firstly, as expected, BP universality was extended to higher dimensions by Balister, Bollob\'as, Morris and Smith. In particular, \cite{Balister22} established that update families such that every open hemisphere contains an open set of stable directions is subcritical. One may recover \cref{th:Toom} in any dimension (this result in any dimension is exactly the content of \cite{Toom80}) from \cite{Balister22}*{Corollary 1.6} in the same way that we deduced \cref{th:Toom} from \cite{Balister16}*{Theorem 9}. However, the multiscale renormalisation of \cite{Balister22} is arguably more complex than the Peierls argument of \cite{Toom80} (see also \cite{Swart22}), making this alternative proof less appealing. Similarly, when restricted to families $\cX$ contained in a half-space, \cite{Balister22} gives an alternative proof of the most difficult part of \cref{th:orientation:BP}.

Secondly, based on a recent generalisation of Toom's approach due to Swart, Szab\'o and Toninelli \cite{Swart22}, Szab\'o and the author \cite{Hartarsky22Toom} gave an alternative proof of the main result of \cite{Balister22} cited above not necessarily restricted to families contained in a half-space, unlike \cref{th:orientation:BP}. To that end they employed a connection between PCA and BP complementary to the one of \cref{prop:correspondence}. Moreover, using both connections simultaneously, they established improved quantitative bounds in Toom's setting of CA with death. 

\let\d\oldd
\let\k\oldk
\let\l\oldl
\let\L\oldL
\let\o\oldo
\let\O\oldO
\let\r\oldr
\let\t\oldt
\let\u\oldu

\bibliographystyle{plain}
\bibliography{Bib}
\end{document}